\documentclass[11pt]{amsart}

\usepackage{amsmath,amssymb,latexsym,soul,cite,mathrsfs}

\usepackage{color,enumitem,graphicx}
\usepackage[colorlinks=true,urlcolor=blue,
citecolor=red,linkcolor=blue,linktocpage,pdfpagelabels,
bookmarksnumbered,bookmarksopen]{hyperref}
\usepackage[english]{babel}

\usepackage[left=2.9cm,right=2.9cm,top=2.8cm,bottom=2.8cm]{geometry}
\usepackage[hyperpageref]{backref}

\usepackage[colorinlistoftodos]{todonotes}
\makeatletter
\providecommand\@dotsep{5}
\def\listtodoname{List of Todos}
\def\listoftodos{\@starttoc{tdo}\listtodoname}
\makeatother

\numberwithin{equation}{section}

\newtheorem{theorem}{Theorem}[section]
\newtheorem{proposition}[theorem]{Proposition}
\newtheorem{lemma}[theorem]{Lemma}
\newtheorem{corollary}[theorem]{Corollary}

\newtheorem{remark}{Remark}
\newtheorem{definition}{Definition}[section]

\begin{document}

\title[Global existence and blow-up of solutions for a parabolic equation ...]
{Global existence and blow-up of solutions for a parabolic equation involving the fractional $p(x)$-Laplacian }
\author{Tahir Boudjeriou}
\address[Tahir Boudjeriou]
{\newline\indent  
	Department of Mathematics
	\newline\indent Faculty of Exact Sciences
	\newline\indent Lab. of Applied Mathematics
	\newline\indent University of Bejaia, Bejaia, 6000, Algeria 
	\newline\indent
	e-mail:re.tahar@yahoo.com}

\pretolerance10000


\begin{abstract}
	\noindent In this paper, we consider a non-local diffusion equation involving the fractional $p(x)$-Laplacian with nonlinearities of variable exponent type. Employing the sub-differential approach we establish the existence of local solutions. By combining the potential well theory with the Nehari manifold, we obtain the existence of global solutions and finite time blow-up of solutions. Moreover, we study the asymptotic stability of global solutions as time goes to infinity in some variable exponent Lebesgue spaces.
 \end{abstract}

\thanks{}
\maketitle	
\section{Introduction} 

In this paper, we are interested in the global existence and finite time blow-up of solutions  to the following problem 
\begin{equation}\label{eq1}\left\{
\begin{array}{llc}
u_{t}+(-\Delta)^{s}_{p(x)}u=|u|^{q(x)-2}u & \text{in}\ & \Omega,\;t>0 , \\
u =0 & \text{in} & \mathbb{R}^{N}\backslash \Omega,\;t > 0, \\
u(x,0)=u_{0}(x), & \text{in} &\Omega , 
\end{array}\right.
\end{equation}
where $\Omega\subset \mathbb{R}^{N}$ ($N\geq 1$) is a smooth bounded domain, $s\in \left(0,1\right)$, $p$ and $q$ are continuous functions that satisfy some technical conditions, which will be mentioned later on. The operator $(-\Delta)_{p(x)}^{s}$ is defined by 
\begin{equation}\label{F.L}
\mathcal{L}u(x)=(-\Delta )_{p(x)}^{s}u(x)=2\lim_{\epsilon \downarrow 0}\int_{\mathbb{R}^{N}\backslash B_{\epsilon}(x)} \frac{|u(x)-u(y)|^{p(x,y)-2}(u(x)-u(y))}{|x-y|^{N+sp(x,y)}}\,dy,\;\;\; x\in \mathbb{R}^{N},
\end{equation}
where $B_{\epsilon}(x)$ denotes the ball of $\mathbb{R}^{N}$ centreted at $x\in \mathbb{R}^{N}$ and radius $ \epsilon >0.$ To our best knowledge, this operator was first introduced by \textit{Kaufmann}, \textit{Rossi} and \textit{Vidal} \cite{UK}, in which the authors extended the Sobolev spaces with variable exponents to the fractional case together with a compact embedding theorem. As an application, they proved the existence and uniqueness of weak solutions for the following fractional  $p(x)$-Laplacian problem 
\begin{equation*}\left\{
\begin{array}{llc}
(-\Delta)^{s}_{p(x)}u+|u|^{q(x)-2}u=f(x) & \text{in}\ & \Omega, \\
u =0 & \text{in} & \partial \Omega, \\
 \end{array}\right.
\end{equation*}
where $f\in L^{a(x)}(\Omega)$ for some $a(x)>1$. Throughout the paper, we will write  
$$ \bar{p}(x)=p(x,x), \;\;\forall x\in \overline{\Omega}.$$
Let us denote by $Q$ the set 
$$ Q= \mathbb{R}^{N}\times\mathbb{R}^{N} \backslash \mathcal{C}\Omega \times \mathcal{C}\Omega, \;\;\text{where}\;\; \mathcal{C}\Omega=\mathbb{R}^{N}\backslash \Omega.$$ 
In this paper, we are assuming the following conditions on the functions $p$ and $q$ : 
\begin{equation*}\label{e1}
2\leq p^{-}=\min_{(x,y)\in \overline{Q}}p(x,y)\leq p(x,y)\leq p^{+}=\max_{(x,y)\in\overline{Q}}p(x,y)< +\infty,
\leqno{(a_{1})}
\end{equation*}

\begin{equation*}\label{e2}
p\;\text{is symmetric, that is}, \,\; p(x,y)=p(y,x)\;\;\forall (x,y)\in \overline{Q},\leqno{(a_{2})}
\end{equation*}
\begin{equation*}\label{e3}
p^{+}< q^{-}=\min_{x\in \overline{\Omega}}q(x)\leq q(x)\leq q^{+}=\max_{x\in \overline{\Omega}}q(x)<\frac{p_{s}^{*}(x)}{2}+1,\leqno{(a_{3})}
\end{equation*}
\begin{equation*}\label{e4}
sp^{+}< N, \;\;\,\leqno{(a_{4})}
\end{equation*}
where $p^{*}_{s}(x)=\frac{N\bar{p}(x)}{N-s\bar{p}(x)}$ is the critical exponent in the fractional Sobolev inequality.\par 

We refer the reader to (\cite{Bh}, \cite{Ez}, \cite{DN}) and the references therein for recent results on the fractional $p(x)$-Laplacian and for further details on the fractional Sobolev spaces with variable exponents. We remark that the operator $(-\Delta)_{p(x)}^{s}u$ is a fractional version of the $p(x)$-Laplacian operator, given by $\text{div}\left( |\nabla u|^{p(x)-2}\nabla u\right)$ associated with the variable exponent Sobolev space.

It worth pointing out that the study of partial differential equations involving $p(x)-$growth conditions has attracted the attention of researchers in recent years. The interest in studying such problems relies not only on mathematical purpose but also on their significance on real models, as explained in (\cite{radrep}, \cite{T.C}, \cite{LD}) and the references therein. 
 In particular, in the theory of elasticity and in mechanics of fluids, more precisely, in fluids of electrorheological type, whose equation of motion is given by 
  $$ u_{t}+\text{div}S(u)+u.\nabla u=f-\nabla \pi,$$ 
  where $u :\mathbb{R}^{3+1}\rightarrow \mathbb{R}$ is the velocity of the fluid at a point in space-time, $\nabla=(\partial_{1},\partial_{2},\partial_{3} )$ the gradient operator,  $\pi :\mathbb{R}^{3+1}\rightarrow \mathbb{R}$ the pressure, $f :\mathbb{R}^{3+1}\rightarrow \mathbb{R}$ represent external forces and $S$ is the stress tensor $S : W_{loc}^{1,1}\rightarrow \mathbb{R}^{3\times 3}$ given by 
  $$ S(u)(x)=\mu(x)\left(1+|D(u(x))|^{\frac{p(x)-2}{2}}\right)D(u(x)),$$
  where $D(u)=\frac{1}{2}\left( \nabla u+(\nabla u)^{T}\right)$ is the symmetric part of the gradient of $u$. Note that if $p(x)=2$, then this equation reduces to the usual Nevier-Stokes equation. If $s=1$, then (\ref{eq1}) reduces to the following problem :
  \begin{equation*}\left\{
  \begin{array}{llc}
  u_{t}-\text{div}(|\nabla u|^{p(x)-2}\nabla u)=|u|^{q(x)-2}u & \text{in}\ & \Omega,\;t>0 , \\
  u =0 & \text{in} & \partial\Omega,\;t > 0, \\
  u(x,0)=u_{0}(x), & \text{in} &\Omega , 
  \end{array}\right.\leqno{(P_{1})}
  \end{equation*}
 In \cite{GA}, by using the sub-differential approach, \textit{Akagi} and \textit{Mastsuura} obtained the well-posedness of solutions for $(P_{1})$ with $f(x,t)$ instead of $|u|^{q(x)-2}u$. Moreover, the large-time behavior of solutions also are considered. We refer the reader to see the papers (\cite{Alv}, \cite{Mes}, \cite{Akg}, \cite{X.L}, \cite{LD1}) and the references therein for some results on global existence and blow-up of solutions for problem $(P_{1})$. In \cite{Ant} \textit{Antontsev}, \textit{Chipot} and \textit{shmarev}, considered the doubly nonlinear parabolic equation with anisotropic variable exponent 
 $$ u_{t}-\text{div}(a(x,t,u)|u|^{\alpha(x,t)}|\nabla u|^{p(x,t)-2}\nabla u)=f(x,t),\;\;\text{in}\; \Omega\times (0,T).$$
 and established conditions on the data which guarantee the comparison principle and uniqueness of bounded weak solutions in suitable Orlicz-Sobolev spaces. If $p(x,y)=p$ and $q(x)=q$, then the problem (\ref{eq1}) becomes 
 \begin{equation*}\left\{
 \begin{array}{llc}
 u_{t}+(-\Delta)_{p}^{s}u=|u|^{q-2}u & \text{in}\ & \Omega,\;t>0 , \\
 u =0 & \text{in} & \mathbb{R}^{N}\backslash\Omega,\;t > 0, \\
 u(x,0)=u_{0}(x), & \text{in} &\Omega , 
 \end{array}\right.\leqno{(P_{2})}
 \end{equation*}
 where $(-\Delta)_{p}^{s}$ is the fractional $p$-Laplacian which is nonlinear nonlocal operator defined on smooth functions by 
 \begin{eqnarray*}
 (-\Delta)_{p}^{s}\varphi(x) &=& 2\lim_{\epsilon \downarrow 0} \int_{\mathbb{R}^{N}\backslash B_{\epsilon}(x)}\frac{|\varphi(x)-\varphi(y)|^{p-2}(\varphi(x)-\varphi(y))}{|x-y|^{N+sp}}\,dy.
 \end{eqnarray*} 
Note that the operator $(-\Delta)_{p}^{s}$ is degenerate when $p>2$ and singular when $1<p<2$. In the case when $p=2$,  \textit{Applebaum} in \cite{J.B} stated that the fractional Laplacian operator of the form $(-\Delta)^{s}$, $0<s<1$, is an infinitesimal generator of stable L\'evy processes. For more details on the fractional Laplacian, see for example (\cite{DV}, \cite{Cl}, \cite{L.C}, \cite{Gi}) and the references therein. 

We point out that in the last years many authors have obtained important results on the fractional $p-$Laplacian in bounded or unbounded domains, for example see  (\cite{Jac}, \cite{Jun}, \cite{Wr}, \cite{Wrr}, \cite{Jac}) and the references therein. \textit{Maz\'on}, \textit{Rossi} and \textit{Toledo} \cite{JMM} considered a model of fractional diffusion involving a nonlocal version of the $p$-Laplacian operator
$$ u_{t}+(-\Delta)_{p}^{s}u=0,\;\;\text{in}\;\;\Omega,\;t>0,\leqno{(P_{3})}$$
where $\Omega \subseteq \mathbb{R}^{N}$, $N\geq 1$, $p\in (1,\infty)$ and $s\in (0,1)$. The authors obtained the existence and uniqueness of strong solutions of the equation $(P_{3})$ by using the sub-differential approach. Moreover, the large-time behavior of solutions also are considered. It is proved, also, that when $s\rightarrow 1^{-}$ with $p\neq 2$ the equation $(P_{3})$ reduces to the well-known $p$-Laplacian evolution equation $u_{t}-\Delta_{p}u=0,$ after inserting a normalizing constant. With the help of potential well theory, \textit{Fu} and \textit{Pucci} \cite{FU}, studied the existence of global weak solutions and established the vacuum isolating and blow-up of strong solutions for the following class of problem 
  $$
  \left\{\begin{array}{l}
  u_{t}+ (-\Delta)^{s} u=|u|^{p-2}u, \;\; x\in \Omega, \; t>0, \\
  u(x,t)=0, \;\;\;\; x\in \mathbb{R}^{N}\backslash \Omega, \;t> 0, \\
  u(x,0)=u_{0}(x), \;\;\; x\in \Omega
  \end{array}\right.\leqno{(M_{2})}
  $$
  for $s\in (0,1)$, $N>2s$ and $2< p\leq 2_{s}^{*}=2N/(N-2s).$  In (\cite{Hng}, \cite{MII}) by combining the Galerkin method with the potential well theory, the authors have studied  the existence of global weak solutions for the degenerate Kirchhoff-type diffusion problems involving fractional Laplacian. Moreover, they obtained also estimates for the lower and upper bounds of the blow-up time.
  
 To the best of our knowledge, there are a few papers that dealt with the existence of elliptic equations involving fractional $p(x)$-Laplacian, see for example (\cite{UK}, \cite{Bh}, \cite{Ez}, \cite{DN}) and the references therein where the authors have used the direct method of calculus of variation and mountain pass theorem to study the existence and multiplicity of solutions.

However, to the author best knowledge, there are no papers to deal with the global existence and blow-up results for problem like (\ref{eq1}). Inspired by the above works, in the present manuscript, we study the existence of global solutions that vanish at infinity or solutions that blow-up in finite time for problem (\ref{eq1}) at high initial energy level.
 
The main difficulty of this problem arises from the fact of working with this new nonlocal fractional $p(x)$-Laplacian operator involving variable exponents.

The rest of paper is organized as follows. In section $2$, we state the main results of this paper. In section $3$, we introduce some important preliminary results which we require throughout the paper. In section $4$, we present some properties involving the functional $E$ restricts to the Nehari manifold $\mathcal{N}$. In sections $5$, via the sub-differential approach we prove the existence of local solutions to the problem (\ref{eq1}). In section $6$, we show the existence of global solutions by combining the potential well theory with the Nehari manifold. Moreover, we show that these solutions vanish at infinity in some $L^{r(.)}(\Omega)-$spaces. In section $7$, by virtue of a differential inequality technique, we prove that the local solutions blow-up in finite time with arbitrary negative initial energy and suitable initial values. 
\section{Main results } 
In this section, we will present the main results of this paper. The energy functional $E : W_{0}\rightarrow \mathbb{R}$ associated with problem $\eqref{eq1}$ is given by 
\begin{equation}\label{eqq1}
E(u) =\int_{Q}\frac{1}{p(x,y)} \frac{|u(x,t)-u(y,t)|^{p(x,y)}}{|x-y|^{N+sp(x,y)}}\,dxdy-\int_{\Omega} \frac{1}{q(x)}|u|^{q(x)}\,dx,
\end{equation}
where the space $W_{0}$ will be introduced in section $3$. By using $(a_{3})$ and theorem \ref{th1} one can verify that $E$ is of calss $C^{1}(W_{0},\mathbb{R})$ and 

$$ E'(u)u=I(u)=\int_{Q}\frac{|u(x,t)-u(y,t) |^{p(x,y)}}{|x-y|^{N+sp(x,y)}}\,dxdy -\int_{\Omega} |u|^{q(x)}\,dx,\quad \quad u\in W_{0}.$$

The potential well associated with problem $\eqref{eq1}$ is the set 
$$ \mathcal{W}:=\{u\in W_{0},\;E(u)< d, \; I(u)> 0\}\cup \{0\}, $$
where $d$ is the depth of the potenial well.
The exterior of the potential well is the set 
$$ Z:=\{u\in W_{0},\;E(u)< d, \; I(u)< 0\}.$$
Related to the functional $E$, we have the well-known Nehari manifold
$$\mathcal{N}=\{u\in W_{0}\backslash \{0\}:\; I(u)=0\}.$$ We define 
\begin{equation}\label{s1}
d=\inf_{u\in \mathcal{N}} E(u).
\end{equation}
It is important to point out that the potential well method was introduced in (\cite{DHH}) to obtain the global existence for nonlinear hyperbolic equations. The most important and typical work on the potential well is due to Payne and Sattinger in \cite {Py} where the authors have studied the initial boundary value problem of semilinear hyperbolic equations and semilinear parabolic equations.
We refer the reader to \cite{I}, where the potential well method was extended to obtain global existence and nonexistence results for the parabolic equations. To state the main results, we need the following two definitions.
\begin{definition}(Strong solution)\label{de1}
	A function $u\in C([0,T],H)$ will be called a strong solution of $\eqref{eq1}$ in $[0,T]$ if 
	\begin{enumerate}
		\item $\frac{du}{dt}\in L^{2}(0,T, H), $
		\item there exist functions $f, g\in L^{2}(0,T,H)$ such that $\frac{du(t)}{dt}+f(t)-g(t)=0, $ $f(t)\in \partial \varphi(u(t))$, $g(t)\in \partial \phi(u(t))$ a.e in $[0,T]$.
		\item the initial condition in $\eqref{eq1}$ is satisfied. 
	\end{enumerate}
\end{definition}
\begin{definition}(Maximal existence time)\label{de2} Let $u(t)$ be a strong solution of problem $\eqref{eq1}$. We define the maximal existence time $T_{\max}$ of $u$ as follows : 
	\begin{enumerate}
		\item If $u$ exists for all $ 0\leq t< +\infty$, then $T_{\max}=+\infty$; 
		\item If there exits a $t_{0}\in (0, +\infty)$ such that $u$ exits for $0\leq t< t_{0}$, but does not exist at $t=t_{0}$, then $T_{\max}=t_{0}$.
	\end{enumerate}
	
\end{definition}
Based on the above preparations, the main results of this paper are the following theorems.
\begin{theorem}\label{th11}
	Let $u_{0}\in W_{0}$ and the assumptions $(a_{1})-(a_{4})$ hold. Then there exists a postive constant $T>0$ such that the problem $\eqref{eq1}$ has a strong solution $u(x,t)$ on $\Omega \times [0,T]$ in the sense of definition \ref{de1}. Furthermore, $u(x,t)$ satisfies the energy equality  
	\begin{equation}\label{en1}
	\int_{0}^{t}\|u_{s}(s)\|_{2}^{2}\,ds+E(u(t))= E(u_{0}),\;\; t\in [0,T].
	\end{equation}
\end{theorem}
\begin{theorem}\label{th3}
	Let $u_{0}\in W_{0}$ and the assumptions $(a_{1})-(a_{4})$ hold. Suppose that 
	\begin{equation}\label{CN}
	E(u_{0})< d\quad\text{and}\quad I(u_{0})> 0.
	\end{equation}
	Then the problem $\eqref{eq1}$ admits a global strong solution such that 
	\begin{equation}
	u(t)\in \mathcal{W}\;\;\text{for}\;\; 0\leq t< \infty,
	\end{equation}
	and satisfying the energy equality  
	\begin{equation}\label{eng1}
	\int_{0}^{t}\|u_{s}(s)\|_{2}^{2}\,ds+E(u(t))= E(u_{0}), \; a.e.\; t\geq 0.
	\end{equation}
	Moreover, 
	\begin{equation}\label{Dec}
	\|u(t)\|_{r(x)}\rightarrow 0,\;\text{as}\; t\rightarrow +\infty,\;\forall r\in (1, p_{s}^{*}).
	\end{equation}
\end{theorem}
\begin{theorem}\label{th4}
	Let $u_{0}\in Z$ and the assumptions $(a_{1})-(a_{4})$ hold. Then 
	\begin{equation}
	u(t)\in Z, \;\forall t\in [0, T_{\max}), 
	\end{equation}
	and $u$ satisfies the energy equality 
	\begin{equation}\label{INN}
	\int_{0}^{t}\|u_{s}(s)\|^{2}_{2}\,ds+E(u(t)) =E(u_{0}), \;\;\;\forall t\in[0,T_{\max}).
	\end{equation}
	Moreover if $E(u_{0})< 0$, then $T_{\max}< +\infty$.	
\end{theorem}
 
\section{Preliminary Results}
In this section, we review definitions for the Lebesgue and fractional Sobolev spaces with variable exponent and some interesting properties of the fractional $p(x)$-Laplacian $\mathcal{L}$ that will be useful to discuss the problem (\ref{eq1}). Besides that, we recall some properties of the subdifferentials.
\subsection{Variable exponent Lebesgue spaces}
The basic properties of the variable exponent Lebesgue spaces can be found in  (\cite{LD}, \cite{J}, \cite{radrep}) and the references therein.\\
Throughout this subsection, without further mentioning, we always assume that $\Omega$ denotes an open set in $\mathbb{R}^{N}$ and $s\in (0,1)$. Denote 

$$ C^{+}(\Omega)=\{ h\in C(\overline{\Omega}),\; h(x)> 1\;\;\text{for all}\; x\in \overline{\Omega}\}.$$
For any $h\in  C^{+}(\Omega)$, we define 
$$ h^{-}=\min_{x\in \overline{\Omega}}h(x), \quad h^{+}=\max_{x\in \overline{\Omega}}h(x).$$
For every $h\in  C^{+}(\Omega)$, the variable exponent Lebesgue space $L^{h(.)}(\Omega)$ is defined by 
$$ L^{h(.)}(\Omega)=\left\{u:\Omega\rightarrow \mathbb{R},\; \exists \lambda> 0, \; \rho_{h(.)}\left( \frac{u}{\lambda}\right)< +\infty\right\},  $$
where the mapping $\rho_{h(.)}: L^{h(.)}(\Omega)\rightarrow \mathbb{R}$ defined by 
$$\rho_{h(.)} (u)=\int_{\Omega} |u(x)|^{h(x)}\,dx,$$
called the $h(.)$-modular of the Lebesgue space $L^{h(.)}(\Omega)$. The space $L^{h(.)}(\Omega)$ endowed with the Luxemburg norm, 
$$\|u\|_{h(.)}=\|u\|_{L^{h(.)}(\Omega)}=\inf\left\{\lambda > 0, \; \rho\left( \frac{u}{\lambda}\right)\leq 1\right\}, $$
is separable and reflexife Banach space. Note that, when $h$ is constant, the Luxemburg norm $\|\,.\,\|_{L^{h(.)}(\Omega)}$ coincide with the standard norm $\|\,.\,\|_{h}$ of the Lebesgue space $L^{h}(\Omega)$.
 
Denoting by $L^{h'(.)}(\Omega)$ the dual space of space  $L^{h(.)}(\Omega)$, where $\frac{1}{h(x)}+\frac{1}{h'(x)}=1$, for any $u\in L^{h(.)}(\Omega)$ and $v\in L^{h'(.)}(\Omega)$ we have the H\"older inequality
\begin{equation}\label{I1}
\int_{\Omega} |uv|\,dx\leq \left(\frac{1}{h^{-}}-\frac{1}{h^{+}}\right)\|u\|_{L^{h(.)}(\Omega)}\|v\|_{L^{h'(.)}(\Omega)},
\end{equation}
If $h_{1}(x), h_{2}(x)\in C^{+}(\overline{\Omega})$ and $h_{1}(x)\leq h_{2}(x)$ for all $x\in \overline{\Omega}$, then $L^{h_{2}(.)}(\Omega)\hookrightarrow L^{h_{1}(.)}(\Omega)$ and the embedding is continuous.\\
In the following proposition, we give some results regarding the relationship between the Luxemburg norm and the $h(.)$-modular mapping.
\begin{proposition}\label{p1}
	For $u\in L^{h(.)}(\Omega)	$ and $\{u_{n}\}_{n\in \mathbb{N}}\subset L^{h(.)}(\Omega)$, we have 
	\begin{enumerate}
		\item $\|u\|_{L^{h(.)}(\Omega)}\leq 1\Rightarrow \|u\|_{L^{h(.)}(\Omega)}^{h^{+}}\leq \rho_{h(.)}(u)\leq \|u\|_{L^{h(.)}(\Omega)}^{h^{-}}, $
		\item $\|u\|_{L^{h(.)}(\Omega)}	\geq1\Rightarrow \|u\|_{L^{h(.)}(\Omega)}^{h^{-}}\leq \rho_{h(.)}(u)\leq \|u\|_{L^{h(.)}(\Omega)}^{h^{+}}, $
		\item $\lim_{n\rightarrow +\infty}\|u_{n}\|_{L^{h(.)}(\Omega)}=0\Leftrightarrow \lim_{n\rightarrow +\infty}\rho_{h(.)}(u_{n})=0,$
		\item $\lim_{n\rightarrow +\infty}\|u_{n}\|_{L^{h(.)}(\Omega)}=+\infty\Leftrightarrow \lim_{n\rightarrow +\infty}\rho_{h(.)}(u_{n})=+\infty.$
	\end{enumerate}
\end{proposition}
\begin{proposition}\label{p2}
	It holds that 
	$$	\sigma^{-}\left(\|u\|_{L^{h(.)}(\Omega)}\right)\leq \int_{\Omega} |u|^{h(x)}\,dx\leq\sigma^{+}\left(\|u\|_{L^{h(.)}(\Omega)}\right)$$
	for all $u\in L^{h(.)}(\Omega)$, with $\sigma^{-}(\tau):=\min(\tau^{h^{-}}, \tau^{h^{+}})$ and $\sigma^{+}(\tau):=\max(\tau^{h^{-}}, \tau^{h^{+}})$ for $\tau \geq 0$.
\end{proposition}
\subsection{Fractional Sobolev spaces with variable exponent }
Let $p :\overline{\Omega} \times \overline{\Omega} \rightarrow (1, +\infty)$ be a continuous satisfies the conditions $(a_{1})-(a_{2})$. We define the fractional Sobolev space with variable exponent via the Gagliardo approach as follows :
$$ W=W^{s, p(x,y)}(\Omega)=\left\{
u\in L^{\bar{p}(.)}(\Omega), \;\; \int_{\Omega \times\Omega} \frac{|u(x)-u(y)|^{p(x,y)}}{\lambda^{p(x,y)}|x-y|^{N+sp(x,y)}}\,dxdy<+\infty, \;\text{for some}\; \lambda >0\right\}, $$
and we set 
$$ [u]_{s, p(x,y)}=\inf\left\{\lambda > 0, \; \int_{\Omega \times\Omega} \frac{|u(x)-u(y)|^{p(x,y)}}{\lambda^{p(x,y)}|x-y|^{N+sp(x,y)}}\,dxdy\leq 1\right\}, $$
the variable exponent Gagliardo-Slobodetskii seminorm. It is already known that (see \cite{UK}) $W$ is a Banach space  with the norm 
$$ \|u\|_{W}=\|u\|_{L^{\bar{p}(.)}(\Omega)}+[u]_{s,p(x,y)}, $$
\begin{theorem}[\cite{Ez}, Theorem 2.1]\label{th1} Let $\Omega $ be a smooth bounded domain in $\mathbb{R}^{N}$ and let $s\in (0,1)$. Let $p :\overline{\Omega}\times \overline{\Omega} \rightarrow (1, +\infty)$ be a continuous function satisfies the conditions $(a_{1})-(a_{2})$ and $sp(x,y)< N$ for all $(x,y)\in \overline{\Omega}\times \overline{\Omega}$. Let $r : \overline{\Omega} \rightarrow (1,+\infty)$ be a continuous function such that 
	$$ p_{s}^{*}(x)	=\frac{N\bar{p}(x)}{N-s\bar{p}(x)}> r(x)\geq r^{-}=\min_{x\in \overline{\Omega}}r(x)> 1\;\;\;\forall x\in \overline{\Omega}.$$
	Then, there exists a constant $C=C(N,s,p,r,\Omega)> 0$ such that, for any $u\in W$, 
	$$ \|u\|_{L^{r(.)}(\Omega)}	\leq C\|u\|_{W}.$$
	Thus, the space $W$ is continuously embedded in $L^{r(x)}(\Omega)$ for any $r\in (1, p_{s}^{*})$. Moreover, this embedding is compact.
\end{theorem}
\begin{lemma}[\cite{Bh}, Lemma 3.1]\label{lem3} Suppose that $\Omega \subset \mathbb{R}^{N}$ is a bounded open domain. Furthermore, assume that $\eqref{e1}$ and $\eqref{e2}$ hold. Then $W$ is a separable and reflexive space.
\end{lemma}
\begin{remark}
	We define the subspace $W_{0}$ of $W$ by 
	$$ W_{0}=\left\{ u\in W,\;\; u=0\;\text{in}\; \mathbb{R}^{N}\backslash \Omega\right\}, $$
	
	which is the closure of $C^{\infty}_{0}(\Omega)$ in $ W$, that is, 
	$$ W_{0}=\overline{C^{\infty}_{0}(\Omega)}^{\|.\|_{W}}.$$
	Then 
	\begin{enumerate}
		\item Theorem \ref{th1} remains true if we replace $W$ by $W_{0}$.
		\item Theorem \ref{th1} implies that $[.]_{s, p(x,y)}$ is a norm on $W_{0}$, which is equivalent to the norm $\|.\|_{W}.$
		\item $W_{0}$ is a Banach space with the norm 
		$$ \|u\|_{W_{0}}=[u]_{s, p(x,y)}=\inf\left\{\lambda > 0, \; \int_{Q} \frac{|u(x)-u(y)|^{p(x,y)}}{\lambda^{p(x,y)}|x-y|^{N+sp(x,y)}}\,dxdy\leq 1\right\}.$$	
	\end{enumerate}
\end{remark}
An imprtant role in manipulating the fractional Soblev-Slobodteskii spaces with variable exponent is palyed by the $(s, p(.,.))$-convex modular function $\rho_{s,p(.,.)}:W_{0}\rightarrow \mathbb{R}.$ defined by 
$$ \rho_{s,p(.,.)}(u)=\int_{Q}\left( \frac{|u(x)-u(y)|}{|x-y|^{s}}\right)^{p(x,y)}\frac{dxdy}{|x-y|^{N}}.$$
\begin{proposition}[\cite {Ez}, Lemma 2.2 and Remark 2.2]\label{Pr}
	For $u\in W_{0}$ and $\{u_{n}\}_{n\in \mathbb{N}} \subset W_{0}$, we have 
	\begin{enumerate}
		\item $[u]_{s, p(.,.)}\leq 1\Rightarrow [u]_{s, p(.,.)}^{p^{+}}\leq \rho_{s, p(.,.)}(u)\leq [u]_{s, p(.,.)}^{p^{-}},$
		\item $[u]_{s, p(.,.)}\geq  1\Rightarrow [u]_{s, p(.,.)}^{p^{-}}\leq \rho_{s, p(.,.)}(u)\leq [u]_{s, p(.,.)}^{p^{+}},$
		\item $\lim_{n\rightarrow +\infty}[u_{n}]_{s, p(.,.)}=0\Leftrightarrow  \lim_{n\rightarrow+\infty}\rho_{s,p(.,.)}(u_{n})=0, $
		\item $\lim_{n\rightarrow +\infty}[u_{n}]_{s, p(.,.)}=+\infty\Leftrightarrow  \lim_{n\rightarrow+\infty}\rho_{s,p(.,.)}(u_{n})=+\infty, $
	\end{enumerate}
\end{proposition}
\begin{proposition}\label{p3}
	It holds that 
	$$	\sigma^{-}\left([u]_{s,p(.,.)}\right)\leq \rho_{s,p(.,.)}(u)\leq\sigma^{+}\left([u]_{s,p(.,.)}\right)$$
	for all $u\in W_{0}$, with $\sigma^{-}(\tau):=\min(\tau^{h^{-}}, \tau^{h^{+}})$ and $\sigma^{+}(\tau):=\max(\tau^{h^{-}}, \tau^{h^{+}})$ for $\tau \geq 0$.	
\end{proposition}
\subsection{Properties of the fractional $p(x)$-Laplacian}
In order to define the fractional $p(x)$-Laplacian for Dirichlet boundary conditions we consider the operator $\mathcal{L} :W_{0}\rightarrow W^{*}_{0}$ defined by 
$$ \langle \mathcal{L}(u), v\rangle=\int_{Q} \frac{|u(x)-u(y)|^{p(x,y)-2}(u(x))-u(y))(v(x)-v(y))}{|x-y|^{N+sp(x,y)}}\,dxdy,\;\;\;\forall u,v\in W_{0},$$
where $W_{0}^{*}$ is the dual space of $W_{0}$. We consider also the following functional
$$ I_{1}(u)=\int_{Q}\frac{1}{p(x,y)} \frac{|u(x)-u(y)|^{p(x,y)}}{|x-y|^{N+sp(x,y)}}\,dxdy,\;\;\;\forall u\in W_{0}.$$
\begin{proposition}[see \cite{Bh} ]\label{pr3} We have, 
	\begin{enumerate}
		\item the functional $I_{1}$ is well defined on $W_{0}$. Moreover $I\in C^{1}(W_{0}, \mathbb{R})$ with the derivative given by 
		$$\langle I'(u), v \rangle=\langle \mathcal{L}(u), v\rangle,\;\;\forall u,v\in W_{0}.$$
		\item the operator $\mathcal{L} :W_{0}\rightarrow W^{*}_{0}$ is a bounded and strictly monotone operator.
		\item the operator $\mathcal{L} :W_{0}\rightarrow W^{*}_{0}$ is a mapping of type $(S_{+})$, i.e, if $u_{n}\rightharpoonup u $ in $W_{0}$ and 
		$$ \limsup_{n\rightarrow +\infty}\langle L(u_{n})-L(u), u_{n}-u\rangle\leq 0, $$
		then $u_{n}\rightarrow u$ in $W_{0}$.
		\item the operator $\mathcal{L} :W_{0}\rightarrow W^{*}_{0}$ is a homeomorphism.
		\item the operator $\mathcal{L} :W_{0}\rightarrow W^{*}_{0}$ is hemicontinuous.
	\end{enumerate}
\end{proposition}
\begin{remark}\label{RRR}
According to \cite[Example 2.3.7, p.26]{H}, the operator $\mathcal{L}_{H}$, the realization of $\mathcal{L}$ at $H=L^{2}(\Omega)$ defined by 

$$ \left\{
\begin{array}{l}
D(\mathcal{L}_{H})=\{u\in W_{0}, \; \mathcal{L}(u)\in H\}, \\
\mathcal{L}_{H}(u)=\mathcal{L}(u)\;\text{if}\; u\in D(\mathcal{L}_{H}),
\end{array}
\right.$$
is a maximal monotone in $H$.
\end{remark}
\subsection{Subdifferentials}
In this subsection, we introduce some useful properties of subdifferentials of proper, convex, and lower semi-continuous functional on a Hilbert space.\par 
Let $H$ be a Hilbert space with the inner product $(., .)$ and the norm $\|\;.\;\|_{2}$. 	
For a functional $\varphi :H \rightarrow (-\infty, +\infty]$, we shall write 
$$ D(\varphi, r)=\{u\in H,\,\varphi(u)\leq r\}\;\;\text{for}\; r\in \mathbb{R}\;\;\text{and}\; D(\varphi)=\bigcup_{r\in \mathbb{R}} D(\varphi, r).$$
Let $\varphi :H\rightarrow (-\infty, +\infty]$ the subdifferential $\partial \varphi $ of $\varphi$ is defined by 
$$ \partial \varphi(u)=\{f\in H, \; \varphi(v)-\varphi(u)\geq (f, v-u),\;\;\forall v\in H\}. $$
It is well known that the subdifferential $\partial \varphi$ is a maximal monotone operator and $D(\partial \varphi)\subset D(\varphi)$.\\
The following type chain rule of subdifferentials is taken from [\cite{H}, Lemma $3.3$, p. 73].
\begin{lemma}\label{lem2}
	Let $\varphi :H\rightarrow (-\infty, +\infty]$ be a proper, convex and lower semicontinuous functional. For some $T> 0$, let $u\in W^{1,1}	(0,T,H)$ and $u(t)\in D(\partial \varphi)$ a.e in $[0,T].$ If there exists a function $f\in L^{2}(0,T,H)$ such that $f(t)\in \partial \varphi(u(t))$ a.e in $[0,T]$,  then the function $t\rightarrow \varphi(u(t))$ is absolutely continuous on $[0,T]$ and 
	$$ \frac{d}{dt}\varphi(u(t))=\left(f(t), \frac{du(t)}{dt}\right), \;\;\text{a.e. in}\; [0,T].$$
\end{lemma}
In what follows, let us recall a very important result found in [\cite{I}, Theorem $3.4$, p. 297].
\begin{theorem}\label{th33}
	Under the following assumptions :
	\begin{enumerate}
		\item the set $D(\varphi, r)$ is compact in $H$ for any $r\in \mathbb{R}$,
		\item $D(\varphi )\subset D(\phi),$
		\item the set $\{(\partial \phi)^{0}(u),\;\;u\in D(\varphi, r)\}$ is bounded in $H$ for any $r\in \mathbb{R}$.

	\end{enumerate}
	where $(\partial \phi)^{0}$ is the unique element of least norm.
	Then, for each $h\in D(\varphi)$, there exist $T> 0$ and a strong solution of this initial value problem 
	\begin{equation*}\label{eq2}\left\{
	\begin{array}{llc}
	\frac{du(t)}{dt}+\partial\varphi(u(t)) -\partial \phi(u(t))\ni 0, \;\;\text{in}\ &H,\; 0< t< T, \\
	u(0)=h, & &
	\end{array}\right.
	\end{equation*}	
\end{theorem}
\section{Properties involving the functional E restricts to $\mathcal{N}$}
For simplicity, in this section, we consider the problem  $\eqref{eq1}$ in the stationary case. We point out that if we replace $u$ in this section by $u(t)$ for any $t\in [0, +\infty)$, all the facts are still valid.\\
In what follows, we denote by $\Lambda$ the constant in the Sobolev embedding $W_{0}\hookrightarrow L^{q(.)}(\Omega)$ i.e, 
\begin{equation}\label{y1}
\Lambda =\inf\left\{\frac{\|u\|_{W_{0}}}{\|u\|_{q(x)}}, \; u\in W_{0}, \; u\neq 0\right\}.
\end{equation}
such infimum is attained for some $v\in W_{0}$, due to the compactness of the embedding $W_{0}\hookrightarrow L^{q(.)}(\Omega)$.
\begin{lemma}\label{lem4}
	Let $u\in W_{0}\backslash \{0\}$. Then there exist unique $\lambda>0$ such that 
	$$
	E'(\lambda u)(\lambda u)=0 \quad \mbox{and} \quad E(\lambda u)=\max_{t \geq 0}E(tu).
	$$
\end{lemma}
\begin{proof} Setting the function 
	$$
	h(t)=E(tu), \quad  t \geq 0,
	$$	
	it is possible to show that $h \in C^1([0,+\infty),\mathbb{R}), h(0)=0, h(t)>0$ for $t$ small enough and $h(t) \to -\infty$ as $t \to +\infty$. Hence, there is $\lambda>0$ such that 
	$$
	h'(\lambda)=0 \quad \mbox{and} \quad h(\lambda)=\max_{t \geq 0}h(t),
	$$
	which is equivalent to 
	$$
	E'(\lambda u)(\lambda u)=0 \quad \mbox{and} \quad E(\lambda u)=\max_{t \geq 0}E(tu).
	$$
	Now, we are going to prove the uniqueness of $\lambda$. Setting $w=\lambda u$, we have that 
	$$
	E'(w)(w)=0 \quad \mbox{and} \quad E(w)=\max_{t \geq 0}E(tw)
	$$
	It is enough to prove that there is no $\lambda_1 \not = 1$ verifying 
	$$
	E'(\lambda_1 w)(\lambda_1 w)=0 \quad \mbox{and} \quad E(\lambda_1 w)=\max_{t \geq 0}E(tw).
	$$
	Assume by contradiction that there is $\lambda_1>1$ satisfying the above inequalities. Then, $E'(\lambda_1 w)(\lambda_1 w)=0 $ leads to 
	$$
	\int_{Q}\frac{|w(x,t)-w(y,t) |^{p(x,y)}}{|x-y|^{N+sp(x,y)}}\,dxdy \geq  {\lambda_1}^{q^{-}-p_{+}}\int_{\Omega} |w|^{q(x)}\,dx.
	$$
	Since 
	$$
	\int_{Q}\frac{|w(x,t)-w(y,t) |^{p(x,y)}}{|x-y|^{N+sp(x,y)}}\,dxdy  = \int_{\Omega} |w|^{q(x)}\,dx,
	$$
	it follows that
	$$
	1 \geq {\lambda_1}^{q^{-}-p_{+}}>1,
	$$
	which is absurd. A similar argument works to prove that $\lambda_1<1$ also does not hold. 
\end{proof}
As a byproduct of the last lemma is the corollary
\begin{corollary} 
	We have $\mathcal{N}\neq \emptyset$. If $u \in \mathcal{N}$, then $\lambda=1$ is the unique number that satisfies 
	$$
	E'(\lambda u)(\lambda u)=0 \quad \mbox{and} \quad E(\lambda u)=\max_{t \geq 0}E(tu).
	$$	
\end{corollary}

\begin{lemma}\label{ee}
	Let $u\in W_{0}	\backslash \{0\}$ and $I(u)=0$. Then 
	\begin{equation}\label{ez}
	d=\inf_{u\in \mathcal{N}}E(u)>\left(\frac{1}{p^{+}}-\frac{1}{q^{-}}\right)R>0,
	\end{equation}
	where $R=\max\left(\Lambda^{q^{+}\left(\frac{q^{+}}{p^{-}}-1\right)},\Lambda^{q^{+}\left(\frac{q^{+}}{p^{+}}-1\right)},\Lambda^{q^{-}\left(\frac{q^{-}}{p^{-}}-1\right)}, \Lambda^{q^{-}\left(\frac{q^{-}}{p^{+}}-1\right)} \right)$.
\end{lemma}
\begin{proof}
	We have 
	\begin{equation}\label{ez20}
	\int_{Q}\frac{|u(x,t)-u(y,t) |^{p(x,y)}}{|x-y|^{N+sp(x,y)}}\,dxdy =\int_{\Omega} |u|^{q(x)}\,dx, 
	\end{equation}

	proposition \ref{p1} combined with $\eqref{ez20}$ yields 
	\begin{equation}\label{IR1}
	\int_{Q}\frac{|u(x,t)-u(y,t) |^{p(x,y)}}{|x-y|^{N+sp(x,y)}}\,dxdy\leq \|u\|_{q(x)}^{q^{-}},
	\end{equation}
	or 
	\begin{equation}\label{IR2}
	\int_{Q}\frac{|u(x,t)-u(y,t) |^{p(x,y)}}{|x-y|^{N+sp(x,y)}}\,dxdy\leq \|u\|_{q(x)}^{q^{+}}.
	\end{equation}
	From $\eqref{y1}$ we have 
	\begin{equation}\label{EM}
	\|u\|_{q(x)}\leq \Lambda^{-1}\|u\|_{W_{0}}, \;\;\forall u\in W_{0}\backslash \{0\},
	\end{equation}
	Combining $\eqref{IR1}-\eqref{EM}$, we deduce
	\begin{equation}\label{IR3}
	\int_{Q}\frac{|u(x,t)-u(y,t) |^{p(x,y)}}{|x-y|^{N+sp(x,y)}}\,dxdy\leq \Lambda^{-q^{-}}\|u\|_{W_{0}}^{q^{-}}
	\end{equation}
	or 
	\begin{equation}\label{IR4}
	\int_{Q}\frac{|u(x,t)-u(y,t) |^{p(x,y)}}{|x-y|^{N+sp(x,y)}}\,dxdy\leq \Lambda^{-q^{+}}\|u\|_{W_{0}}^{q^{+}}
	\end{equation}
	Using proposition \ref{Pr}, it turns out that 
	\begin{equation}\label{INN1}
	\int_{Q}\frac{|u(x,t)-u(y,t) |^{p(x,y)}}{|x-y|^{N+sp(x,y)}}\,dxdy\leq \Lambda^{-q^{-}}\rho_{s, p(x,y)}^{\frac{q^{-}}{p^{-}}}(u)\;\;\; \text{or}\;\;\;\int_{Q}\frac{|u(x,t)-u(y,t) |^{p(x,y)}}{|x-y|^{N+sp(x,y)}}\,dxdy\leq \Lambda^{-q^{-}}\rho_{s, p(x,y)}^{\frac{q^{-}}{p^{+}}}(u),
	\end{equation}
	and 
	\begin{equation}\label{INN2}
	\int_{Q}\frac{|u(x,t)-u(y,t) |^{p(x,y)}}{|x-y|^{N+sp(x,y)}}\,dxdy\leq \Lambda^{-q^{+}}\rho_{s, p(x,y)}^{\frac{q^{+}}{p^{-}}}(u)\;\;\; \text{or}\;\;\;\int_{Q}\frac{|u(x,t)-u(y,t) |^{p(x,y)}}{|x-y|^{N+sp(x,y)}}\,dxdy\leq \Lambda^{-q^{+}}\rho_{s, p(x,y)}^{\frac{q^{+}}{p^{+}}}(u).
	\end{equation} 
	Thus, $\eqref{INN1}$ and $\eqref{INN2}$ ensure that 
	$$\int_{Q}\frac{|u(x,t)-u(y,t) |^{p(x,y)}}{|x-y|^{N+sp(x,y)}}\,dxdy\geq R. $$
	On the other hand, it is easy to see that
	\begin{equation}\label{ea2}
	E(u)\geq \left(\frac{1}{p^{+}}-\frac{1}{q^{-}}\right)\rho_{s, p(x,y)}(u).
	\end{equation}
	
	Since $\frac{1}{p^{+}}-\frac{1}{q^{-}}> 0$, we have 
	$$ E(u)\geq \left(\frac{1}{p^{+}}-\frac{1}{q^{-}}\right)R.$$
	Hence, the proof is now complete.
\end{proof}
\begin{lemma}
	There exist an extremal of the variation problem $\eqref{s1}$. More precisely, there is a function $w \in \mathcal{N}$ such that $E(w)=d$	
\end{lemma}
\begin{proof}
	Let $\{u_{k}\}^{\infty}_{k=1}\subset \mathcal{N}$ be a minimizing sequence for $E$ such that 
	$$ \lim_{k\rightarrow \infty}E(u_{k})=d.$$
	From $\eqref{ea2}$, $ \{u_{k}\}^{\infty}_{k=1}$ is bounded in $W_{0}$. Since the embedding  $W_{0}\hookrightarrow L^{q(.)}(\Omega)$ is compact, there exists a function $u$ and a subsequence of $ \{u_{k}\}^{\infty}_{k=1}$, still denoted by $\{u_{k}\}^{\infty}_{k=1}$, such that 
	$$ u_{k}\rightharpoonup u \;\text{in}\; W_{0},$$
	$$ u_{k}\rightarrow u \;\text{in}\; L^{q(.)}(\Omega),$$
	$$ u_{k}\rightarrow u \;\text{a.e. in }\; \Omega.$$
	Then by using the weak lower semicontinuity of the norm in $W_{0}$, we get 
	\begin{eqnarray*}
		E(u)&=&\int_{Q}\frac{1}{p(x,y)} \frac{|u(x,t)-u(y,t)|^{p(x,y)}}{|x-y|^{N+sp(x,y)}}\,dxdy-\int_{\Omega} \frac{1}{q(x)}|u|^{q(x)}\,dx,\\
		&\leq & \liminf_{k\rightarrow\infty}\left( \int_{Q} \frac{1}{p(x,y)}\frac{|u_{k}(x,t)-u_{k}(y,t)|^{p(x,y)}}{|x-y|^{N+sp(x,y)}}\,dxdy-\int_{\Omega} \frac{1}{q(x)}|u_{k}|^{q(x)}\,dx\right)\\
		&=&\liminf_{k\rightarrow\infty} E(u_{k})=d.
	\end{eqnarray*} 
	Thanks to $u_{k}\in \mathcal{N}$ one has $u_{k}\in W_{0}\backslash \{0\}$ and $I(u_{k})=0$, then by lemma \ref{ee} it follows that 
	$$ \rho_{q(x)}(u_{k})\geq \left(\frac{1}{p^{+}}-\frac{1}{q^{-}}\right)R.$$
	Hence, by strong convergence in $L^{q(.)}(\Omega)$ it turns out that $\rho_{q(x)}(u)\neq 0$, thus $u\in W_{0}\backslash \{0\}.$ On the other hand, the equality $I(u_k)=0$ for all $k \in \mathbb{N}$ ensures that $I(u) \leq 0$. From Lemma \ref{lem4}, there is a unique $\lambda>0$ such that 
	$$
	E'(\lambda u)(\lambda u)=0 \quad \mbox{and} \quad E(\lambda u)=\max_{t \geq 0}E(tu).
	$$	
	As  $I(u)=E'(u)u\leq 0$, we can conclude that $\lambda \in (0,1]$. Since $\lambda u \in \mathcal{N}$ and $\lambda \in (0,1]$, we derive that 
	\begin{equation} \label{EQ1}
	d \leq E(\lambda u)=E(\lambda u)-\frac{1}{q^{-}}E'(\lambda u)(\lambda u).
	\end{equation}
	Since 
	$$
	\begin{array}{l}
	E(\lambda u)-\frac{1}{q^{-}}E'(\lambda u)(\lambda u)=\displaystyle \int_{Q} \left(\frac{1}{p(x,y)}-\frac{1}{q^{-}}\right)\frac{\lambda^{p(x,y)}|u(x,t)-u(y,t)|^{p(x,y)}}{|x-y|^{N+sp(x,y)}}\,dxdy+ \\
	\mbox{} \\ \hspace{5 cm} \displaystyle \int_{\Omega} \left(\frac{1}{q^{-}}-\frac{1}{q(x)}\right)\lambda^{q(x)}|u|^{q(x)}\,dx,
	\end{array}
	$$
	using the fact that $\lambda \in (0,1]$, we get 
	\begin{equation} \label{EQ2*}
	E(\lambda u)-\frac{1}{q^{-}}E'(\lambda u)(u)E(\lambda u) \leq E (u)-\frac{1}{q^{-}}E' (u)(u)
	\end{equation} 	
	Now, using the fact that $u_k \rightharpoonup u$ in $W_0$, we deduce 
	\begin{equation} \label{EQ2}
	E(u)-\frac{1}{q^{-}}E'(u)(u) \leq \liminf_{k\rightarrow\infty}\left(E(u_k)-\frac{1}{q^{-}}E'(u_k)(u_k) \right)=\liminf_{k\rightarrow\infty}E(u_k)=d.
	\end{equation}
	From (\ref{EQ1}),(\ref{EQ2*}) and (\ref{EQ2}),  $E(\lambda u)=d$. Setting  $w=\lambda u$, we have that $w \in\mathcal{N}$ and $E(w)=d$,  showing the desired result.	
\end{proof}
\begin{remark} Note that in fact the proof above implies that $\lambda=1$, otherwise we will get a strict inequality in (\ref{EQ1}), which will lead to a new contradiction. 
\end{remark}	
\section{Local existence} 
In this section we give the proof of Theorem \ref{th11}. Here we will convert
 the system $\eqref{eq1}$ to a first order Cauchy problem of an abstract evolution equation in $H=L^{2}(\Omega)$. To this end, we define the functionals $\varphi :H\rightarrow (-\infty, +\infty]$ and $\phi:H\rightarrow (-\infty, +\infty]$ as follows 
\begin{equation}\label{f1}
\varphi(u)=\left\{
\begin{array}{lcc}
\int_{Q}\frac{1}{p(x,y)}\frac{|u(x,t)-u(y,t)|^{p(x,y)}}{|x-y|^{N+sp(x,y)}}\,dxdy, &\text{if} & u\in W_{0},\\ 
+\infty  & \text{if} &  u\in H\backslash W_{0}.
\end{array}
\right.
\end{equation}
\begin{equation}\label{f3}
\phi(u)=\left\{
\begin{array}{lcc}
\int_{\Omega }\frac{1}{q(x)}|u(x,t)|^{q(x)}\,dx, &\text{if} & u\in L^{q(.)}(\Omega),\\ 
+\infty  & \text{if} &  u\in H\backslash L^{q(.)}(\Omega) .
\end{array}
\right.
\end{equation}
It is easy to check that the functionals $\varphi$ and $\phi$ are proper, convex and lower semicontinuous. In what follows we claim that 
\begin{equation}\label{e5}
\partial \varphi(u) = (-\Delta)_{p(x)}^{s}u, \;\;\;\;\partial \phi(u)=|u|^{q(x)-2}u.
\end{equation}
Indeed, according to remark \ref{RRR} we have that $\mathcal{L}_{H}$ and $\partial \varphi$ are both maximal monotone operators in $H$. Thus it is enough to prove that $ \mathcal{L}_{H}(u)\subset \partial \varphi(u).$ Let $u\in D(\mathcal{L}_{H})$ and $v=\mathcal{L}_{H}(u)$, then for all $w\in W_{0}$ 
$$ \langle  v, w-u\rangle_{W^{*}_{0}, W_{0}}=\langle \mathcal{L}_{H}(u), w-u\rangle_{W^{*}_{0}, W_{0}}$$
$$=\int_{Q } |u(x)-u(y)|^{p(x,y)-2}\frac{(u(x)-u(y))((w-u)(x)-(w-u)(y))}{|x-y|^{N+sp(x,y)}}\,dxdy.$$ 
Now by using this inequality 
\begin{equation}\label{INq}
\bar{p}(x)|r|^{\bar{p}(x)-2}r(s-r)\leq |s|^{\bar{p}(x)}-|r|^{\bar{p}(x)}, \;\; \forall r,s\in \mathbb{R}\;\;\text{and}\;\;\forall \bar{p}\in C^{+}(\Omega), 
\end{equation}
we deduce 
\begin{eqnarray}\label{AAA}
 \langle  v, w-u\rangle_{W^{*}_{0}, W_{0}} &\leq& \int_{Q} \frac{1}{p(x,y)}\frac{|w(x)-w(y)|^{p(x,y)}-|u(x)-u(y)|^{p(x,y)}}{|x-y|^{N+sp(x,y)}}\,dxdy,\\\nonumber
 &=& \varphi(w)-\varphi(u).
\end{eqnarray}
If $w\in H\backslash W_{0}$ we have $\varphi(w)=+\infty$ and thus $\eqref{AAA}$ clearly holds. Therefore, we have shown that $\mathcal{L}_{H}(u)=v\in \partial \varphi(u).$ This concludes that 
$\partial\varphi (u)=\mathcal{L}_{H}(u).$ Proceeding in the same way one can show that $\partial\phi(u)=|u|^{q(x)-2}u$. By virtue of $\eqref{e5}$, the system $\eqref{eq1}$ can be rewritten as the following asbtract Cauchy problem 
\begin{equation}\label{eq3}\left\{
\begin{array}{lc}
\frac{du(t)}{dt}+\partial\varphi(u(t)) -\partial \phi(u(t))\ni 0\;\text{in}\;H,\; 0< t< T, \\
u(0)=u_{0}, & 
\end{array}\right.
\end{equation}
\begin{lemma}
	The set $D(\varphi, r)$ is compact in $L^{2}(\Omega)$ for any $r\in \mathbb{R}$. Moreover, $D(\varphi )\subset D(\phi).$
\end{lemma}
\begin{proof} This is an immediate consequence of Theorem \ref{th1}.
\end{proof}

\begin{lemma} The set $\{(\partial \phi)^{0}(u),\;\;u\in D(\varphi, r)\}$ is bounded in $L^{2}(\Omega)$ for any $r\in \mathbb{R}$.
	
\end{lemma}
\begin{proof} It is enough to consider $r>0$.  Since $\phi \in  C^{1}(L^{q(.)}(\Omega),\mathbb{R})$, a simple computation gives that there is unique $f_u \in L^{2}(\Omega)$ such that $(\partial \phi)(u)=\{f_u\}$, and so, $(\partial \phi)^{0}(u)=\{f_u\}$. Moreover, 
	$$
	\int_{\Omega}|u|^{q(x)-2}uv\,dx=(f_u,v)_{L^{2}(\Omega)}, \quad \forall v \in C_0^{\infty}(\Omega). 
	$$
	The above equality yields $|u(x)|^{q(x)-2}u(x)=f_u(x)$ a.e. in $\Omega$. From this, 
	$$
	|(f_u,v)_{L^{2}(\Omega)}| \leq \||u(x)|^{q(x)-1}\|_{L^{2}(\Omega)}\|v\|_{L^{2}(\Omega)}, \quad \forall v \in C_0^{\infty}(\Omega).
	$$
	A simple computation shows that 
	\begin{equation}\label{TTd}
	\int_{\Omega}|u(x)|^{2(q(x)-1)}\,dx \leq C\sigma^{+}([u]_{s, p(x,y)}).
	\end{equation}
Using the fact that $u \in D(\varphi, r)$,  there is $C=C(r)>0$ such that 
	$$
	\sigma^{+}([u]_{s, p(x,y)})\leq C, \quad \forall u \in D(\varphi,r).
	$$
	This combined with (\ref{TTd}) yields 
	$$
	\||u(x)|^{q(x)-1}\|_{L^{2}(\Omega)} \leq C, \quad \forall u \in D(\varphi,r).
	$$
	Hence
	$$
	\|f_u\|_{L^{2}(\Omega)} =\sup_{\|v\|_{L^{2}(\Omega)}\leq 1}|(f_u,v)_{L^{2}(\Omega)}|\leq C, \quad \forall u \in D(\varphi,r),
	$$
	showing the lemma.   	
\end{proof}
Therefore, the well-posedness of the abstract evolution problem (\ref{eq3}) follows from Theorem \ref{th33}. Moreover, from Lemma \ref{lem2}, we deduce
	\begin{equation}\label{eqqq1}
	u\in C([0,T], W_{0}), 
	\end{equation} 
	and 
	$$\int_{0}^{t}\|u_{s}(s)\|_{2}^{2}\,ds+E(u(t))=E(u_{0}),\;\;t\in[0,T].$$ 
Hence, the proof is now complete.
\section{Global existence and asymptotic behavior}
In this section by combining the potential well theory with the Nehari manifold, we prove that the local solutions of problem (\ref{eq1}) can be extended to global solutions, see (\cite{CT}, \cite{CT2} \cite{FU}, \cite{Hng}) and the references therein for some results on the existence of  global  solutions.

	From $\eqref{en1}$, we have 
	\begin{equation}\label{eq7}
	E(u(t))\leq E(u_{0})< d, \;\;\forall t\in [0, T_{\max}).
	\end{equation}
	Assume that $u(t)\neq 0$ for all $t\in [0, T_{\max})$, otherwise if $u(t)=0$ for some $t\geq 0$ we get a contradiction with $\eqref{CN}$. Now we claim that 
	\begin{equation}\label{ezz}
	I(u(t))> 0\;\;\;\text{ for all }\;t\in [0, T_{\max}).
	\end{equation}
	
	Assume that $\eqref{ezz}$ is not true, then there exists $t^{*}\in (0,T_{\max})$ such that $I(u(t^{*}))\leq 0$. On the other hand, we have $I(u_{0})> 0$ then by using the intermediate value theorem we obtain $t_{1}\in (0,t^{*})$ such that $I(u(t_{1}))=0$, this implies $u(t_{1})\in \mathcal{N}$. From $\eqref{s1}$, we have 
	$$ E(u(t_{1})) \geq  d, $$
	but this is a contradiction with $\eqref{eq7}$. Hence the claim $\eqref{ezz}$ holds. Combining the fact that  $u(t)\in \mathcal {W}$ with $\eqref{ea2}$ we obtain   
	\begin{equation} \label{eq8}
	\left(\frac{1}{p^{+}}-\frac{1}{q^{-}}\right)\rho_{s, p(x,y)}(u(t))< d, \;\;\;\forall t\in [0,T_{\max}),
	\end{equation}
	and 
	\begin{equation}\label{eq9}
	\int_{0}^{t}\|u_{s}(s)\|_{2}^{2}\,ds< d,\;\;\forall t\in [0,T_{\max}).
	\end{equation}
	From where it follows that  $T_{\max}=+\infty$.
	Now, our task is to show the asymptotic behavior of
	 $u(t)$ in $L^{r(.)}(\Omega)$ when $t\rightarrow +\infty$ . From $\eqref{eq8}$, we have 
	 $$\sup_{t\in (0,+\infty)}\rho_{s, p(x,y)}(u(t))< +\infty. $$
	Since $W_{0}\hookrightarrow L^{r(.)}(\Omega)$ is compact for all $r\in (1, p_{s}^{*})$, then it is known that $\omega$-limit set $\omega(u_{0})$ in the $L^{r(.)} (\Omega)$ topology 
	$$\omega(u_{0}):=\{v\in L^{r(.)}(\Omega), \;\;\exists t_{n}\rightarrow +\infty, \; u(t_{n})\rightarrow v\;\;\text{in}\; L^{r(.)}(\Omega)\}, $$
	is nonempty and consists of equilibria, see \cite[Theorem 4.3.3, p. 91]{D}. We see that for any nontrivial equilibrium $v$, we have $$I(v)=0,$$
 by $\eqref{s1}$, it follows that 
	\begin{equation}\label{aw}
	E(v)\geq d.
	\end{equation}
	Now we claim that $v\notin \omega(u_{0})$. Indeed, assuming that $v\in \omega(u_{0})$ we get 
	$$ u(t_{n}) \rightarrow v\;\;\text{in}\; L^{r(.)}(\Omega), \;\;\text{as}\; n\rightarrow \infty.$$
	and 
	$$ u(t_{n}) \rightharpoonup  v\;\;\text{in}\; W_{0}, \;\;\text{as}\; n\rightarrow \infty.$$
	By  using the weak lower semicontinuity of the norm in $W_{0}$, we obtain 
	$$ E(v)\leq \liminf_{n\rightarrow \infty}E(u(t_{n}))< d,$$
	but this is a contradiction with $\eqref{aw}$. Consequently 
	$$ \|u(t)\|_{r(x)}\rightarrow 0\;\;\text{as}\;\;t\rightarrow +\infty.$$
	Thus, the proof is now complete.
	\section{Blow-up phenomena}
	In this section, by means of a differential inequality technique,  we prove that the local solutions of problem $\eqref{eq1}$ blow-up in finite time.
	
	 We point out that by \cite[Theorem 3.3.4, p.55]{D} the local solution $u$ of problem (\ref{eq1}) can be extended to a maximal solution in $[0, T_{\max})$. Thus, the energy equality $\eqref{INN}$ can be obtained by extending $\eqref{en1}$ to $[0, T_{\max})$. By an argument similar to that in the previous section one can show that 
	\begin{equation}\label{non}
	\text{If}\;u_{0}\in Z\;\text{then }\; u(t)\in Z\;\;\;\forall t\in [0, T_{\max}).
	\end{equation}
	Set $$\varphi(t)=\frac{1}{2}\|u(t)\|_{2}^{2}.$$ 
	Multiplying the equation in $\eqref{eq1}$ by $u$ we obtain 
	\begin{eqnarray*}
		\varphi'(t)&=&\int_{\Omega} u_{t}(t)u(t)\,dx =-\rho_{s, p(x,y)}(u)+\rho_{q(x)}(u),\\
		&\geq & -p^{+}E(u(t))\left(1-\frac{p^{+}}{q^{-}}\right)\rho_{q(x)}(u),\\
		&\geq & -p^{+}E(u_{0})+c_1\rho_{q(x)}(u).
	\end{eqnarray*}
The last inequality combined with Proposition (\ref{p1}) and $E(u_0) < 0$ leads to 
			$$
			\varphi'(t) \geq c_1\min\left\{\|u\|^{q^+}_{q(x)},\|u\|^{q^-}_{q(x)}\right\}.
			$$
			By using the continuous embedding $L^{q(.)}(\Omega) \hookrightarrow L^{2}(\Omega)$, we get 
			$$
			\varphi'(t) \geq \min\left\{C^{q^+}\|u\|^{q^+}_{2},C^{q^-}\|u\|^{q^-}_{2}\right\},
			$$
			for some positive constant $C>0$, and so, 
			\begin{equation} \label{IN1} 
			\varphi'(t) \geq \min\left\{C^{q^+}(\varphi(t))^{q^+/2},C^{q^-}(\varphi(t))^{q^-/2}\right\}:=h(t),
			\end{equation}
			Assume by contradiction that $T_{\max}=+\infty$.  We know that $\varphi(t)>0$ for all $t>0$, otherwise if $u(t)=0$ for some $ t \geq 0$, which is absurd, because of  $\eqref{non}$. Have this in mind, the inequality (\ref{IN1}) ensures $\varphi'(t)>0$ for all $t>0$, from where it follows that $\varphi$ is increasing in $(0,+\infty)$. Hence
			$$
			\varphi(t)-\varphi(0)=\int_{0}^{t}\varphi'(s)\,ds\geq \int_{0}^{t}h(s)\,ds\geq t \min\left\{C^{q^+}(\varphi(0))^{q^+/2},C^{q^-}(\varphi(0))^{q^-/2}\right\},
			$$   
			and so, 
			$$
			\varphi(t) \geq t\min\left\{C^{q^+}(\varphi(0))^{q^+/2},C^{q^-}(\varphi(0))^{q^-/2}\right\}+\varphi(0)\to +\infty, \quad \mbox{as} \quad t \to +\infty.
			$$
			From this, there is $t_0>0$ such that,
			\begin{equation} \label{IN2} 
			\varphi(t) >1, \quad \forall t \geq t_0. 
			\end{equation}
			From (\ref{IN1})-(\ref{IN2}), 
			$$
			\varphi'(t) \geq C^{q^+}(\varphi(t))^{q^+/2}, \quad \forall t \geq t_0.
			$$
			Therefore, 
			$$
			(-q^+/2+1)\frac{\varphi'(t)}{(\varphi(t))^{q^+/2}} \leq C_1, \quad \forall t \geq t_0,
			$$
			where $C_1=(-q^+/2+1)C^{p^+}<0$. Hence, 
			$$
			\int_{t_0}^{t}\frac{d}{ds}(\varphi(t))^{-q^+/2+1}\,ds \leq \int_{t_0}^{t}C_1\,ds=C_1(t-t_0), \quad \forall t \geq t_0
			$$
			that is,
			$$
			(\varphi(t))^{-q^+/2+1}-(\varphi(t_0))^{-q^+/2+1} \leq C_1(t-t_0), \quad \forall t \geq t_0
			$$
			or equivalently
			$$
			(\varphi(t))^{-q^+/2+1}(t)\leq C_1(t-t_0)+(\varphi(t_0))^{-q^+/2+1}(t_0), \quad \forall t \geq t_0.
			$$
			Since $C_1<0$, we have that  
			$$
			C_1(t-t_0)+(\varphi(t_0))^{-q^+/2+1}(t_0) \to -\infty, \quad \mbox{as} \quad t \to +\infty.
			$$
			Thus, there is $t_1>0$ such that 
			$$
			(\varphi(t))^{-q^+/2+1}(t) <0, \quad \forall t \geq t_1, 
			$$
			which is absurd, because $(\varphi(t))^{-q^+/2+1}(t)>0$ for all $t>0$.  This proves $T_{\max}<+\infty$.

\section*{Acknowledgements}
The author would like to thank Professor Claudianor Alves for his suggestions and fruitful discussions. The author would like to thank the anonymous referee for the careful reading of the paper and for his/her valuable comments. 

\end{document}